\documentclass[11pt,a4paper]{article}
\usepackage[a4paper,left=2cm,right=2cm,top=2cm,bottom=2cm]{geometry}

\usepackage[utf8]{inputenc}
\usepackage{amsthm,amsmath,amssymb}
\usepackage[mathcal,mathscr]{eucal}
\usepackage{graphicx,graphics}
\usepackage{enumerate}
\usepackage[sort,nocompress]{cite}
\usepackage{caption}
\usepackage[position=b]{subcaption}
\usepackage{thm-restate}
\usepackage{hyperref}
\hypersetup{
    colorlinks=true,       % false: boxed links; true: colored links
    linkcolor=blue,        % color of internal links (change box color with linkbordercolor)
    citecolor=red,         % color of links to bibliography
    filecolor=magenta,     % color of file links
    urlcolor=cyan,         % color of external links
    linktocpage=true
}

\theoremstyle{plain}
\newtheorem{theorem}{Theorem}
\newtheorem{lemma}[theorem]{Lemma}
\newtheorem{corollary}[theorem]{Corollary}
\newtheorem{claim}[theorem]{Claim}

\theoremstyle{definition}

\providecommand{\keywords}[1]{\textbf{\textit{Keywords:}} #1}
\newcommand{\cI}{\mathcal{I}}
\newcommand\free{\operatorname{{\it free}}}

\title{Complexity of packing common bases in matroids}
\author{
Krist\'of B\'erczi\thanks{MTA-ELTE Egerv\'ary Research Group, Department of Operations Research, E\"otv\"os Lor\'and University, Budapest. Email: \texttt{berkri@cs.elte.hu}.} \and
Tam\'as Schwarcz\thanks{E\"otv\"os Lor\'and University, Budapest. Email: \texttt{schtomi97@caesar.elte.hu}.}
}

\date{}

\begin{document}

\maketitle

\begin{abstract}
One of the most intriguing unsolved questions of matroid optimization is the characterization of the existence of $k$ disjoint common bases of two matroids. The significance of the problem is well-illustrated by the long list of conjectures that can be formulated as special cases, such as Woodall's conjecture on packing disjoint dijoins in a directed graph, or Rota's beautiful conjecture on rearrangements of bases.

In the present paper we prove that the problem is difficult under the rank oracle model, i.e., we show that there is no algorithm which decides if the common ground set of two matroids can be partitioned into $k$ common bases by using a polynomial number of independence queries. Our complexity result holds even for the very special case when $k=2$.

Through a series of reductions, we also show that the abstract problem of packing common bases in two matroids includes the NAE-SAT problem and the Perfect Even Factor problem in directed graphs. These results in turn imply that the problem is not only difficult in the independence oracle model but also includes NP-complete special cases already when $k=2$, one of the matroids is a partition matroid, while the other matroid is linear and is given by an explicit representation.
\end{abstract}
\medskip


\section{Introduction} \label{sec:intro}

Various graph characterization and optimization problems can be treated conveniently by applying the basic tools of matroid theory. The main role of matroid theory is not only that it helps understanding the true background of known problems but they are often unavoidable in solving natural optimization problems in which matroids do not appear explicitly at all. One of the most powerful results is the matroid intersection theorem of Edmonds \cite{edmonds1970submodular} providing a min-max formula for the maximum cardinality of a common independent set of two matroids. In particular, this gives rise to a characterization of the existence of a common basis. The closely related problem of packing bases in one matroid is also nicely solved by Edmonds and Fulkerson \cite{edmonds1965transversals} even in the more general case when there are $k$ matroids on $S$ and we want to pick a basis from each in a pairwise disjoint way.

Edmonds and Giles \cite{edmonds1977min} initiated a common generalization of network flow theory and matroid theory by introducing the notion of submodular flows. Another framework that generalizes matroid intersection, introduced by Frank and Jordán \cite{frank1995minimal}, characterized optimal coverings of supermodular bi-set functions by digraphs and provided a min-max result in which the weighted version includes NP-complete problems. Despite being widely general, none of these frameworks gave answer for the longstanding open problem of finding $k$ disjoint common bases of two matroids. This problem was open even for $k=2$ in the sense that no general answer was known similar to the case of one matroid, but no NP-complete special cases were known either. The few special cases that are settled include Edmonds' theorem on the existence of $k$ disjoint spanning arborescences of a digraph rooted at the same root node \cite{edmonds1973edge},  K\H{o}nig's result on 1-factorization of bipartite graphs \cite{konig1916graphen}, and results of Keijsper and Schrijver \cite{keijsper1998packing} on packing connectors.

There is a long list of challenging conjectures that can be formulated as a statement about packing common bases of two matroids. Rota's beautiful basis conjecture \cite{huang1994relations} states that if $M$ is a matroid of rank $n$ whose ground set can be partitioned into $n$ disjoint bases $B_1,\dots,B_n$, then it is possible to rearrange the elements of these bases into an $n\times n$ matrix in such a way that the rows are exactly the given bases, and the columns are also bases of $M$. Only partial results are known, see e.g. \cite{geelen2006rota,geelen2007rota,chow2009reduction,harvey2011disjoint,bucic2018halfway}. 

Woodall's conjecture \cite{woodall1978menger} on packing disjoint dijoins in a directed graph is also a special case of packing common bases, as was shown by Frank and Tardos \cite{FrankP5}. Given a directed graph $D$, a dijoin is a subset of arcs whose contraction results in a strongly connected digraph. The conjecture states that the maximum number of pairwise disjoint dijoins equals the minimum size of a directed cut. The conjecture was known to be true for $k=2$, for source-sink connected digraphs by Schrijver \cite{schrijver1982min} and independently by Feofiloff and Younger \cite{feofiloff1987directed}, for series-parallel digraphs by Lee and Wakabayashi \cite{lee2001note}. Recently M\'esz\'aros \cite{Meszaros.note} proved that if $k$ is a prime power, then the conjecture holds if the underlying undirected graph is $(k-1,1)$-partition-connected.

The capacitated packing of $k$-arborescences is yet another problem that can be formulated as packing common bases in two matroids \cite{egres_open}. A \emph{$k$-arborescence} is the union of $k$ pairwise edge-disjoint arborescences rooted at the same vertex. Given a directed graph $D=(V,A)$ with arc-capacities $c:A\rightarrow\mathbb{Z}_+$ satisfying $c(a)\leq\ell$ for $a\in A$ and a node $r_0\in V$, the problem asks if the existence of a capacity-obeying packing of $k\ell$ spanning arborescences rooted at $r_0$ implies the existence of a capacity-obeying packing of $\ell$ $k$-arborescences rooted at $r_0$. Although several papers generalizing Edmonds' theorem on packing arborescences appeared in the last decade (for recent papers with great overviews, see e.g. \cite{matsuoka2019reachability,gao2019arborescences,kakimura2018b}), this problem remains widely open. 

This illustrious list of open problems underpins the significance of the abstract, matroidal version. Given two matroids $M_1=(S,\mathcal{I}_1)$ and $M_2=(S,\mathcal{I}_2)$, there are three different problems that can be asked: (A) Can $S$ be partitioned into $k$ common independent sets of $M_1$ and $M_2$? (B) Does $S$ contain $k$ disjoint common bases of $M_1$ and $M_2$? (C) Does $S$ contain $k$ disjoint common spanning sets of $M_1$ and $M_2$? These problems may seem to be closely related, and (A) and (B) are indeed in a strong connection, but (C) is actually substantially different from the others. 

There is an obvious necessary condition for the existence of a partition into $k$ common independent sets: the ground set has to be partitionable into $k$ independent sets in both matroids. Davies and McDiarmid showed that this condition is sufficient for the case of strongly base orderable matroids \cite{davies1976disjoint}. Kotlar and Ziv \cite{kotlar2005partitioning} proved that if $M_1$ and $M_2$ are matroids on $S$ and no element is 3-spanned in $M_1$ or $M_2$, then $S$ can be partitioned into two common independent sets. They conjectured that this can be generalized to arbitrary $k$: if no element is $(k+1)$-spanned in $M_1$ or $M_2$, then $S$ can be partitioned into $k$ common independent sets. Recently, Takazawa and Yokoi proposed a new approach building upon the generalized-polymatroid intersection theorem \cite{takazawa2018generalized}. Their result explains the peculiar condition appearing in the theorem of Kotlar and Ziv on how many times an element is spanned, and they also provide new pairs of matroid classes for which the natural necessary condition is sufficient.

To the best of our knowledge, the time complexity of problems (A), (B) and (C) under the independence oracle model was open until now. (It is worth mentioning that the independence, rank, circuit-finding, spanning, port, strong basis and certain closure oracles are polynomially equivalent \cite{robinson1980computational,hausmann1981algorithmic,coullard1996independence}.) We will concentrate on the \textsc{PartitionIntoCommonBases} problem, defined as follows: Given matroids $M_1=(S,\mathcal{I}_1)$ and $M_2=(S,\mathcal{I}_2)$, find a partition of $S$ into common bases. Note that this problem is a special case of all (A), (B) and (C). Our main contribution is the following. 

\begin{restatable}{theorem}{MainTheorem}\label{thm:main}
The \textsc{PartitionIntoCommonBases} problem requires an exponential number of independence queries. 
\end{restatable} 

We prove the theorem by reduction from a problem that we call \textsc{PartitionIntoModularBases} and seems to be closely related to the matroid parity. Theorem~\ref{thm:main} immediately implies that all three of the problems (A), (B) and (C) are difficult under the rank oracle model. We also verify that the problem is not only difficult in the independence oracle model, but it also includes NP-complete special cases.

\begin{restatable}{theorem}{MainTheoremNP}\label{thm:main2}
\textsc{PartitionIntoCommonBases} includes NP-complete problems. 
\end{restatable} 

The proof of Theorem~\ref{thm:main2} will show that the problem of partitioning into common bases is already difficult in the very special case when $|S|=2r_1(S)=2r_2(S)$, one of the matroids is a partition matroid and the other is a linear matroid given by an explicit linear representation.

\medskip

The rest of the paper is organized as follows. Basic definitions and notation are introduced in Section~\ref{sec:plem}. We introduce the \textsc{PartitionIntoModularBases} problem in Section~\ref{sec:par} and prove its hardness in the independence oracle model. Theorem~\ref{thm:main} is then proved by reduction from \textsc{PartitionIntoModularBases}. In Section~\ref{sec:linear}, we show that \textsc{PartitionIntoModularBases} includes the NP-complete NAE-SAT problem, thus proving Theorem~\ref{thm:main2}. The same proof implies that \textsc{PartitionIntoCommonBases} remains difficult when restricted to linear matroids given by explicit linear representations. Section~\ref{sec:main2} considers the \textsc{PartitionIntoModularBases} problem for transversal matroids. Through a series of reductions that might be of independent combinatorial interest, we show that the NP-complete Perfect Even Factor problem also fits in the framework of packing common bases. Finally, Section~\ref{sec:conclusions} concludes the paper with further remarks and open questions.

\section{Preliminaries} \label{sec:plem}

Matroids were introduced by Whitney \cite{whitney1992abstract} and independently by Nakasawa \cite{nishimura2009lost} as abstract generalizations of linear independence in vector spaces. A matroid $M$ is a pair $(S,\mathcal{I})$ where $S$ is the \textbf{ground set} of the matroid and $\mathcal{I}\subseteq 2^S$ is the family of \textbf{independent sets} that satisfies the following, so-called \textbf{independence axioms}: (I1) $\emptyset\in\mathcal{I}$, (I2) $X\subseteq Y\in \cI\Rightarrow X\in\cI$, (I3) $X,Y\in\cI,|X|<|Y|\Rightarrow\exists e\in Y-X\ \text{s.t.}\ X+e\in\cI$. The \textbf{rank} of a set $X\subseteq S$ is the maximum size of an independent subset of $X$ and is denoted by $r_M(X)$. The maximal independent sets of $M$ are called \textbf{bases}. Alternatively, simple properties of bases can be taken as axioms as well. In terms of bases, a matroid $M$ is a pair $(S,\mathcal{B})$ where $\mathcal{B}\subseteq 2^S$ satisfies the  \emph{basis axioms}: (B1) $\mathcal{B}\neq\emptyset$, (B2) for any $B_1,B_2\in \mathcal{B}$ and $u\in B_1- B_2$ there exists $v\in B_2- B_1$ such that $B_1-u+v\in\mathcal{B}$.

For a set $S$, the matroid in which every subset of $S$ is independent is called a \textbf{free matroid} and is denoted by $M^{\free}_S$. For disjoint sets $S_1$ and $S_2$, the \textbf{direct sum} $M_1\oplus M_2$ of matroids $M_1=(S_1,\cI_1)$ and $M_2=(S_2,\cI_2)$ is a matroid $M=(S_1\cup S_2,\mathcal{I})$ whose independent sets are the disjoint unions of an independent set of $M_1$ and an independent set of $M_2$. The \textbf{$k$-truncation} of a matroid $M=(S,\cI)$ is a matroid $(S,\cI_k)$ such that $\cI_k=\{X\in\cI:|X|\leq k\}$. We denote the $k$-truncation of $M$ by $(M)_{k}$. 

A matroid $M=(S,\cI)$ is called \textbf{linear} (or \textbf{representable}) if there exists a matrix $A$ over a field $\mathbb{F}$ and a bijection between the columns of $A$ and $S$, so that $X\subseteq S$ is independent in $M$ if and only if the corresponding columns in $A$ are linearly independent over the field $\mathbb{F}$. The class of linear matroids includes several well-investigated matroid families such as graphic matroids \cite{tutte1965lectures}, rigidity matroids \cite{graver1991rigidity,whiteley1996some} and gammoids \cite{lindstrom1973vector}. It is not difficult to verify that the class of linear matroids is closed under duality, taking direct sum (when the field $\mathbb{F}$ for linear representations is common), taking minors and taking $k$-truncation. Moreover, if we apply any of these operations for a matroid (or a pair of matroids) given by a linear representation over a field $\mathbb{F}$, then a linear representation of the resulting matroid can be determined by using only polynomially many operations over $\mathbb{F}$ (see e.g. \cite{lokshtanov2018deterministic}).

Given a bipartite graph $G=(S,T;E)$, a set $X\subseteq S$ is independent in the \textbf{transversal matroid} $M=(S,\cI)$ if and only if $X$ can be covered by a matching of $G$. Transversal matroids are also linear as they are exactly the dual matroids of strict gammoids. However, one has to be careful when discussing the complexity of problems related to transversal matroids.
If a transversal matroid $M=(S,\cI)$ is given by an independence oracle, then determining its bipartite graph representation is difficult as it requires an exponential number of independence queries \cite{jensen1982complexity}. If a bipartite graph $G=(S,T;E)$ is given, then a linear representation of the transversal matroid associated with $G$ on the ground set $S$ over the field of fractions $\mathbb{F}(\mathbf{x})$ can be determined in deterministic polynomial time. Nevertheless, such a representation is not suitable for use in efficient deterministic algorithms. Substituting random values for each indeterminate in $\mathbf{x}$ from a field having size large enough leads to a randomized polynomial time algorithm that gives a linear representation over a field where operations can be carried out efficiently \cite{marx2009parameterized}. The derandomization of this approach might require to overcome major obstacles as it would have important consequences in complexity theory \cite{kabanets2004derandomizing}.

A matroid $M=(S,\cI)$ of rank $r$ is called \textbf{paving} if every set of size at most $r-1$ is independent, or in other words, every circuit of the matroid has size at least $r$. Blackburn, Crapo and Higgs \cite{blackburn1973catalogue} enumerated all matroids up to eight elements, and observed that most of these matroids are paving matroids. Crapo and Rota \cite{crapo1970foundations} suggested that perhaps paving matroids dominate the enumeration of matroids. This statement was made precise by Mayhew, Newman, Welsh and Whittle in \cite{mayhew2011asymptotic}. They conjectured that the asymptotic fraction of matroids on $n$ elements that are paving tends to $1$ as $n$ tends to infinity. A similar statement on the asymptotic ratio of the logarithms of the numbers of matroids and sparse paving matroids has been proven in \cite{pendavingh2015number}. We will need the following technical statement \cite{hartmanis1959lattice,welsh2010matroid,frank2011connections}.

\begin{theorem}\label{thm:pav}
Let $r\geq 2$ be an integer and $S$ a set of size at least $r$. Let $\mathcal{H}=\{H_1,\dots,H_q\}$ be a (possibly empty) family of proper subsets of $S$ in which every set $H_i$ has at least $r$ elements and the intersection of any two of them has at most $r-2$ elements. Then the set system $\mathcal{B}_{\mathcal{H}}=\{X\subseteq S:\ |X|=r, X\not\subseteq H_i\ \text{for } i=1,\dots,q\}$ forms the set of bases of a paving matroid. Moreover, every paving matroid can be obtained in this form.
\end{theorem}

Let $M=(S,\cI)$ be a matroid whose ground set is partitioned into two-element subsets called \textbf{pairs}. A set $X\subseteq S$ is called a \textbf{parity set} if it is the union of pairs. The matroid parity problem asks for a parity independent set of maximum size. This problem was introduced by Lawler \cite{lawler1976combinatorial} as a common generalization of graph matching and matroid intersection. Unfortunately, matroid parity is intractable for general matroids as it includes NP-hard problems, and requires an exponential number of queries if the matroid is given by an independence oracle \cite{jensen1982complexity,lovasz1978matroid}. On the positive side, for linear matroids, Lov\'asz developed a polynomial time algorithm \cite{lovasz1978matroid} that is applicable if a linear representation is available. In the next section, we will define a packing counterpart of the matroid parity problem in which the goal is to partition the ground set of a matroid into parity bases.

\section{Hardness in the independence oracle model} \label{sec:par}

Recall that in the matroid parity problem the aim is to find a parity independent set of maximum size. We define an analogous problem regarding partitions of the ground set.

Let $M=(S,\cI)$ be a matroid and let $\mathcal{P}$ be a partition of the ground set into non-empty subsets. Members of $\mathcal{P}$ are called \textbf{modules}, and a set $X\subseteq S$ is \textbf{modular} if it is the union of modules. The \textsc{PartitionIntoModularBases} problem is as follows: Given a matroid $M=(S,\cI)$ over a ground set $S$ of size $2r(S)$ together with a partition $\mathcal{P}$ of $S$, find a partition of $S$ into two modular bases.

%On the one hand, the $NP$-complete problem of finding a $k$-vertex complete subgraph in a given $n$-vertex graph can be transformed into an instance of matroid parity \cite{soto2014simple}. Meantime, deciding if a graph is the union of two $k$-vertex complete subgraphs is an easy exercise. On the other hand, 
%We will show that \textsc{PartitionIntoModularBases} is difficult even for graphic matroids. This is in contrast to the matroid parity problem which is solvable for graphic matroids \cite{lovasz1978matroid}.

In what follows, we prove that \textsc{PartitionIntoModularBases} is intractable for general matroids as it requires an exponential number of independence queries even in the special case when every module is a pair. We will refer to this variant as the \textsc{PartitionIntoParityBases} problem. Although \textsc{PartitionIntoParityBases} seems to be closely related to matroid parity, the relationship between the two problems is unclear.

\begin{theorem} \label{thm:abstract_paired}
The \textsc{PartitionIntoParityBases} problem requires an exponential number of independence queries.
\end{theorem}
\begin{proof}
Let $S$ be a finite set of $4t$ elements and let $\mathcal{P}$ be an arbitrary partition of $S$ into $2t$ pairs, forming the modules. Let $\mathcal{H}=\{X\subseteq S:|X|=2t,X\ \text{is a parity set}\}$. For a parity set $X_0$ with $|X_0|=2t$, define $\mathcal{H}_0=\mathcal{H}-\{X_0,S-X_0\}$. Both $\mathcal{H}$ and $\mathcal{H}_0$ satisfy the conditions of Theorem~\ref{thm:pav}, hence $\mathcal{B}_{\mathcal{H}}$ and $\mathcal{B}_{\mathcal{H}_0}$ define two matroids $M$ and $M_0$, respectively.

Clearly, the ground set cannot be partitioned into parity bases of $M$, while $X_0\cup(S-X_0)$ is such a partition for $M_0$. For any sequence of independence queries which does not include $X_0$ or $S-X_0$, the result of those oracle calls are the same for $M$ and $M_0$. That is, any sequence of queries which does not include at least one of the parity subsets $X_0$ or $S-X_0$ cannot distinguish between $M$ and $M_0$, concluding the proof of the theorem.
\end{proof}

Now we turn to the proof of Theorem~\ref{thm:main}. We will need the following technical lemma.

\begin{lemma} \label{lem:help}
Let $\ell\in\mathbb{Z}_+$ and let $S$ be a ground set of size $9\ell$. There exist two matroids $M'_\ell$ and $M''_\ell$ of rank $5\ell$ satisfying the following conditions:
\begin{enumerate}[(a)]
\item $S$ can be partitioned into two common independent sets of $M'_\ell$ and $M''_\ell$ having sizes $5\ell$ and $4\ell$; \label{it:a}
\item for every partition $S=S_1\cup S_2$ into two common independent sets of $M'_\ell$ and $M''_\ell$, we have $\{|S_1|,|S_2|\}=\{5\ell,4\ell\}$, that is, one of the partition classes has size exactly $5\ell$ while the other has size exactly $4\ell$. \label{it:b}
\end{enumerate} 
\end{lemma}
\begin{proof}

\begin{figure}[t!]
\centering
\begin{subfigure}[b]{\textwidth}
  \centering
  \includegraphics[width=.9\linewidth]{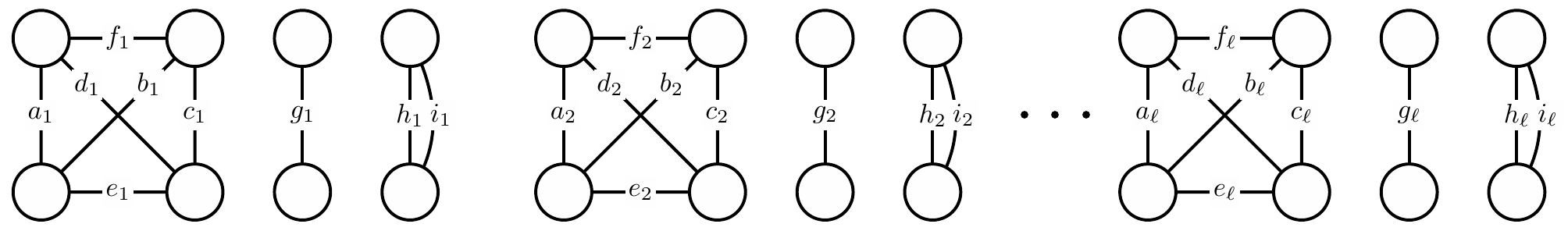}
  \caption{The graph $G'$ corresponding to $M'_\ell$.}
  \label{fig:mat1}
\end{subfigure}\vspace{5pt}
\begin{subfigure}[b]{\textwidth}
  \centering
  \includegraphics[width=.9\linewidth]{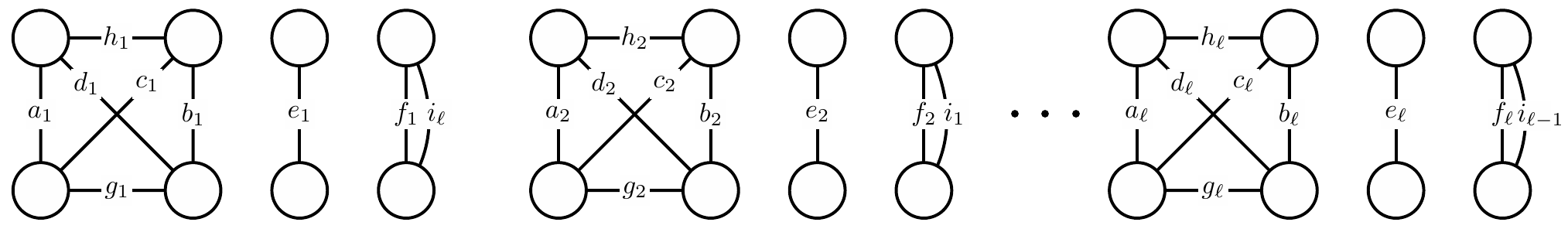}
  \caption{The graph $G''$ corresponding to $M''_\ell$.}
  \label{fig:mat2}
\end{subfigure}
\caption{The edge-labeled graphs defining $M'_\ell$ and $M''_\ell$.}
\label{fig:obs}
\end{figure}

Let $S=\bigcup_{j=1}^\ell W_j$ denote a ground set of size $9\ell$ where $W_j=\{a_j,b_j,c_j,d_j,e_j,f_j,g_j,h_j,i_j\}$. Let $M'_\ell$ and $M''_\ell$ denote the graphic matroids defined by the edge-labeled graphs $G'$ and $G''$ on Figures~\ref{fig:mat1} and \ref{fig:mat2}, respectively. We first prove \eqref{it:a}.

\begin{claim} \label{cl:kk-l}
$S$ can be partitioned into two common independent sets of $M'_\ell$ and $M''_\ell$ having sizes $5\ell$ and $4\ell$.
\end{claim}
\begin{proof}
It is not difficult to find a partition satisfying the conditions of the claim, for example, $S=S_1\cup S_2$ where $S_1=\bigcup_{j=1}^\ell\{d_j,e_j,f_j,g_j,h_j\}$ and $S_2=\bigcup_{j=1}^\ell\{a_j,b_j,c_j,i_j\}$.
\end{proof}

In order to verify \eqref{it:b}, take an arbitrary partition $S=S_1\cup S_2$ into common independent sets of $M'_\ell$ and $M''_\ell$. Let $\hat{W}_j=W_j-i_j$.

\begin{claim} \label{cl:k-xk-l+2}
For each $j=1,\dots,\ell$, $S_1$ and $S_2$ partition $\hat{W}_j$ into common independent sets having sizes $5$ and $3$. Moreover, the elements $e_j,f_j,g_j$ and $h_j$ are contained in the same partition class.
\end{claim}
\begin{proof}
$S_1$ and $S_2$ necessarily partition the $K_4$ subgraphs spanned by $\hat{W}_j$ in $G'$ and $G''$ into two paths of length $3$, so 
%$|S_i\cap\{e_j,f_j\}|=|S_i\cap\{g_j,h_j\}|=1$ for $i=1,2$. This also implies that 
$|S_i\cap \{a_j,b_j,c_j,d_j,e_j,f_j\}|= 3$ and $|S_i\cap \{a_j,b_j,c_j,d_j,g_j,h_j\}|= 3$ for $i=1,2$. This implies that either $|S_i\cap\{e_j,f_j\}|=|S_i\cap\{g_j,h_j\}|=1$ for $i=1,2$, or $e_j,f_j,g_j$ and $h_j$ are contained in the same partition class. 

In the former case, we may assume that $g_j\in S_1$ and $h_j\in S_2$. In order to partition the $K_4$ subgraph spanned by $\hat{W}_j$ in $G''$ into two paths of length $3$, either $\{a_j,b_j\}\subseteq S_1$ and $\{c_j,d_j\}\subseteq S_2$ or $\{c_j,d_j\}\subseteq S_1$ and $\{a_j,b_j\}\subseteq S_2$ hold. However, these sets cannot be extended to two paths of length $3$ in $G'$, a contradiction.

Thus $e_j,f_j,g_j$ and $h_j$ are contained in the same partition class. Since $|S_i \cap \{a_j,b_j,c_j,d_j,e_j,f_j\}|=3$ for $i=1,2$, the claim follows.
\end{proof}

Now we analyze how the presence of edges $i_j$ affect the sizes of the partition classes. By Claim~\ref{cl:k-xk-l+2}, we may assume that $\{e_1,f_1,g_1,h_1\}\subseteq S_1$, and so $i_\ell\in S_2$.

\begin{claim} \label{cl:ij}
$\{e_j,f_j,g_j,h_j\}\subseteq S_1$ and $i_j\in S_2$ for $j=1,\dots,\ell$.
\end{claim}
\begin{proof}
We prove by induction on $j$. By assumption, the claim holds for $j=1$. Assume that the statement is true for $j$. As $i_j$ is parallel to $f_{j+1}$ in $G''$, $f_{j+1}\in S_1$. By Claim~\ref{cl:k-xk-l+2}, $\{e_{j+1},f_{j+1},g_{j+1},h_{j+1}\}\subseteq S_1$. As $i_{j+1}$ is parallel to $h_{j+1}$ in $G'$, necessarily $i_{j+1}\in S_2$, proving the inductive step.
\end{proof}

Claims~\ref{cl:k-xk-l+2} and \ref{cl:ij} imply that $|S_1|=5\ell$ while $|S_2|=4\ell$, concluding the proof of the lemma.
\end{proof}

It should be emphasized that, for our purposes, any pair of matroids satisfying the conditions of Lemma~\ref{lem:help} would be suitable; we defined a specific pair, but there are several other choices that one could work with. 

We are now in the position to prove Theorem~\ref{thm:main}.\footnote{The proof is based on reduction from \textsc{PartitionIntoModularBases}, and so, by Theorem~\ref{thm:abstract_paired}, we could assume that every module has size $2$. However, our construction in Section~\ref{sec:linear} for proving that the linear case is already difficult uses modules of larger sizes, hence we show reduction from the general version of \textsc{PartitionIntoModularBases}.}

\MainTheorem*
\begin{proof}
We prove by reduction from \textsc{PartitionIntoModularBases}. Let $M=(S,\cI)$ be a matroid together with a partition $\mathcal{P}$ of its ground set into modules. Recall that $|S|=2r(S)$, that is, the goal is to partition the ground set into two modular bases.

We define two matroids as follows. For every set $P\in \mathcal{P}$, let $M'_P=(S_P,\mathcal{I}'_P)$ and $M''_P=(S_P,\mathcal{I}''_P)$ be copies of the matroids $M'_{|P|}$ and $M''_{|P|}$ provided by Lemma~\ref{lem:help}. We denote 
\[
S'=S\cup\left(\bigcup_{P\in \mathcal{P}} S_P\right).
\]
Note that $|S'|=10|S|$, that is, the size of the new ground set is linear in that of the original. Let
\begin{align*}
M_1&=\left(M\oplus\left(\bigoplus_{P\in \mathcal{P}} M'_P\right)\right)_{\frac{|S'|}{2}}\\[0.5cm]
M_2&=\bigoplus_{P\in \mathcal{P}}(M^{\free}_P\oplus M''_P)_{5|P|}.
\end{align*}
$M_1$ is defined as the $|S'|/2$-truncation of the direct sum of $M$ and the matroids $M'_P$ for $P\in \mathcal{P}$. For the other matroid, we first take the $5|P|$-truncation of the direct sum of $M''_P$ and the free matroid $M^{\free}_P$ on $P$ for each $P\in \mathcal{P}$, and then define $M_2$ as the direct sum of these matroids. We first determine the ranks of $M_1$ and $M_2$.

\begin{claim} \label{cl:rank}
Both $M_1$ and $M_2$ have rank $|S'|/2$.
\end{claim}
\begin{proof}
The rank of $M_1$ is clearly at most $|S'|/2$ as it is obtained by taking the $|S'|/2$-truncation of a matroid. Hence it suffices to show that $M\oplus\left(\bigoplus_{P\in \mathcal{P}} M'_P\right)$ has an independent set of size at least $|S'|/2$. For each $P\in \mathcal{P}$, let $B_P$ be a basis of $M'_P$. Then $\bigcup_{P\in \mathcal{P}} B_P$ is an independent set of $M_1$ having size $\sum_{P\in\mathcal{P}}5|P|=5|S|=|S'|/2$ as requested.

The rank of $(M^{\free}_P\oplus M''_P)_{5|P|}$ is $5|P|$ for each $P\in \mathcal{P}$. This implies that the rank of $M_2$ is at most $\sum_{P\in\mathcal{P}}5|P|=5|S|=|S'|/2$. We get an independent set of that size by taking a basis $B_P$ of $M''_P$ for each $P\in\mathcal{P}$, and then taking their union $\bigcup_{P\in \mathcal{P}} B_P$.
\end{proof}

The main ingredient of the proof is the following.

\begin{claim} \label{cl:02}
If $S'=B'_1\cup B'_2$ is a partition of $S'$ into two common bases of $M_1$ and $M_2$, then each module $P \in \mathcal{P}$ is contained completely either in $B'_1$ or in $B'_2$.
\end{claim}
\begin{proof}
For an arbitrary module $P$, let $I_1=S_P\cap B'_1$ and $I_2=S_P\cap B'_2$. Clearly, $I_1$ and $I_2$ are independent in both $M'_P$ and $M''_P$. By Lemma~\ref{lem:help}, we may assume that $|I_1|=4|P|$ and $|I_2|=5|P|$. As the rank of $(M^{\free}_P\oplus M''_P)_{5|P|}$ is $5|P|$, we get $P\subseteq B'_1$ as requested.
\end{proof}

The next claim concludes the proof of the theorem.

\begin{claim} \label{cl:equiv}
$S$ has a partition into two modular bases if and only if $S'$ can be partitioned into two common bases of $M_1$ and $M_2$.
\end{claim}
\begin{proof}
For the forward direction, assume that there exists a partition $S'=B'_1\cup B'_2$ of $S'$ into two common bases of $M_1$ and $M_2$. By Claim~\ref{cl:02}, for every module $P\in \mathcal{P}$, the elements of $P$ are all contained either in $B_1$ or in $B_2$. This implies that $B_1=S\cap B'_1$ and $B_2=S\cap B'_2$ are modular sets. By the definition of $M_1$, these sets are independent in $M$. As $|S|=2r(S)$, $B_1$ and $B_2$ are modular bases of $M$. 

To see the backward direction, let $S=B_1\cup B_2$ be a partition of $S$ into modular bases. For each $P\in \mathcal{P}$, let $I^1_P\cup I^2_P$ be a partition of $S_P$ into common independent sets of $M'_P$ and $M''_P$ having sizes $4|P|$ and $5|P|$, respectively. Recall that such a partition exists by Lemma~\ref{lem:help}. Then the sets 
\begin{align*}
B'_1&=B_1\cup\{I^1_P:P\subseteq B_1\}\cup\{I^2_P:P\subseteq B_2\}\qquad\text{and}\\
B'_2&=B_2\cup\{I^1_P:P\subseteq B_2\}\cup\{I^2_P:P\subseteq B_1\}
\end{align*}
form common independent sets of $M_1$ and $M_2$ and partitions the ground set $S'$. By Claim~\ref{cl:rank}, $B'_1$ and $B'_2$ are bases, concluding the proof of the claim.
\end{proof}

The theorem follows by Claim~\ref{cl:equiv}.
\end{proof}

\section{Hardness in the linear case} \label{sec:linear}

The aim of this section is to show that the \textsc{PartitionIntoModularBases} problem might be difficult to solve even when the matroid is given with a consise description, namely by an explicit linear representation over a field in which the field operations can be done efficiently. In order to do so, we consider the \textsc{PartitionIntoModularBases} problem for graphic matroids, called the \textsc{PartitionIntoModularTrees} problem. The problem can be rephrased as follows: Given a graph $G=(V,E)$ and a partition $\mathcal{P}$ of its edge set, find a partition of $E$ into two spanning trees consisting of partition classes.

\begin{theorem} \label{thm:modtrees}
\textsc{PartitionIntoModularTrees} is NP-complete.
\end{theorem}
\begin{proof}
We prove by reduction from Not-All-Equal Satisfiability, abbreviated as \textsc{NAE-SAT}: Given a CNF formula, decide if there exists a truth assignment not setting all literals equally in any clause. It is known that \textsc{NAE-SAT} is NP-complete, see \cite{schaefer1978complexity}.\footnote{In \cite{schmidt2010computational}, Schmidt proved that \textsc{NAE-SAT} remains NP-complete when restricted to the class \textsc{LCNF}$^3_+$, that is, for monotone, linear and $3$-regular formulas. Although the construction appearing in our reduction could be slightly simplified based on this observation, we stick to the case of \textsc{NAE-SAT} as it appears to be a more natural problem.}

Let $\Phi=(U,\mathcal{C})$ be an instance of \textsc{NAE-SAT} where $U=\{x_1,\dots,x_n\}$ is the set of variables and $\mathcal{C}=\{C_1,\dots,C_m\}$ is the set of clauses. We construct an undirected graph $G=(V,E)$ as follows. We may assume that no clause contains a variable and its negation simultaneously, as for such a clause every assignment has a true value and no assignment sets all literals equally. 

\begin{figure}[t!]
\centering
\begin{subfigure}[b]{0.5\textwidth}
  \centering
  \includegraphics[width=.6\linewidth]{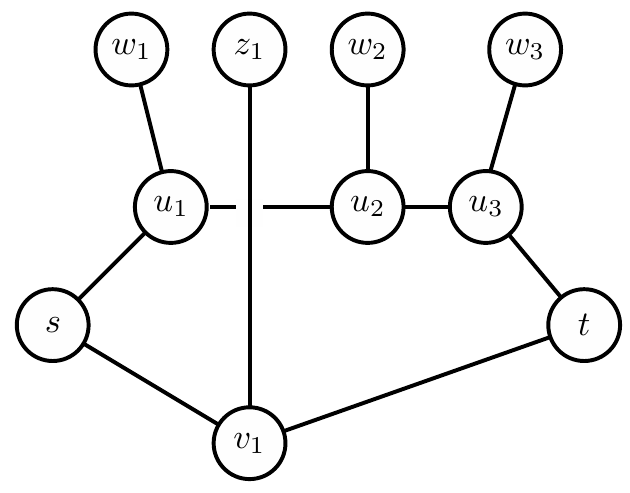}
  \caption{Gadget $H[3,1]$.}
  \label{fig:gadget1}
\end{subfigure}\hfill
\begin{subfigure}[b]{0.5\textwidth}
  \centering
  \includegraphics[width=.6\linewidth]{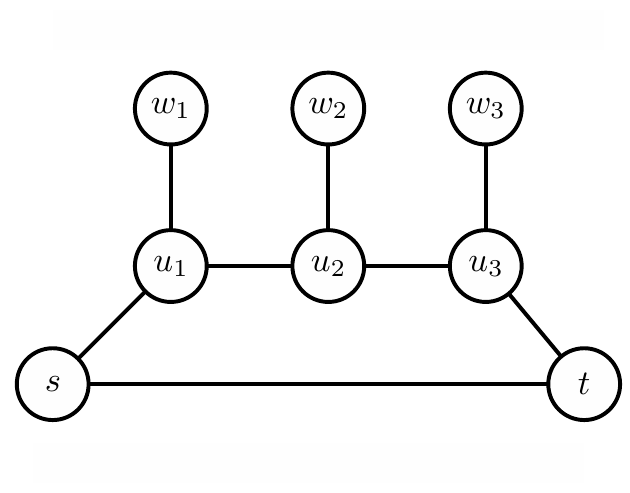}
  \caption{Gadget $H[3,0]$.}
  \label{fig:gadget2}
\end{subfigure}
\caption{Examples for variable gadgets.}
\label{fig:gadget}
\end{figure}

First we construct the variable gadget. Let $H[p,q]$ denote an undirected graph on node set $\{s,t\}\cup\{u_i,w_i:i=1,\dots,p\}\cup\{v_j,z_j:j=1,\dots,q\}$ consisting of the two paths $su_1,u_1u_2,\dots,u_pt$ and $sv_1,v_1v_2,\dots,v_qt$, together with edges $u_iw_i$ for $i=1,\dots,p$ and $v_jz_j$ for $j=1,\dots,q$. If any of $p$ or $q$ is $0$, then the corresponding path simplifies to a single edge $st$ (see Figure~\ref{fig:gadget}).

We construct an undirected graph $G=(V,E)$ as follows. With each variable $x_j$, we associate a copy of $H[p_j,q_j]$ where the literal $x_j$ occurs $p_j$ times and the literal $\bar{x}_j$ occurs $q_j$ times in the clauses. These components are connected together by identifying $t^j$ with $s^{j+1}$ for $j=1,\dots,n-1$. We apply the notational convention that in the gadget corresponding to a variable $x_j$, we add $j$ as an upper index for all of the nodes. For a variable $x_j$, the ordering of the clauses naturally induces an ordering of the occurrences of $x_j$ and $\bar{x}_j$. For every clause $C_i$, we do the following. Assume that $C_i$ involves variables $x_{j_1},\dots,x_{j_{\ell}}$. Recall that no clause contains a variable and its negation simultaneously, hence $\ell$ is also the number of literals appearing in $C_i$. If $C_i$ contains the literal $x_{j_k}$ and this is the $r$th occurrence of the literal $x_{j_k}$ with respect to the ordering of the clauses, let $y^i_{j_k}:=w^{j_k}_r$. If $C_i$ contains the literal $\bar{x}_{j_k}$ and this is the $r$th occurrence of the literal $\bar{x}_{j_k}$ with respect to the ordering of the clauses, let $y^i_{j_k}:=z^{j_k}_r$. Then we add the edges of the cycle $y^i_{j_1},\dots,y^i_{j_\ell}$ to the graph. Finally, we close the construction by adding edges $t^{n}w^j_k$ for $j=1,\dots,n$, $k=1,\dots,p_j$, and adding edges $t^{n}z^j_k$ for $j=1,\dots,n$, $k=1,\dots,q_j$ (see Figure~\ref{fig:graph}). An easy computation shows that the number of edges is $|E|=2|U|+4\sum_{C\in\mathcal{C}}|C|$, while the number of nodes is $|V|=|U|+2\sum_{C\in\mathcal{C}}|C|+1$, that is, $|E|=2|V|-2$.

\begin{figure}[t!]
\centering
\includegraphics[width=0.6\textwidth]{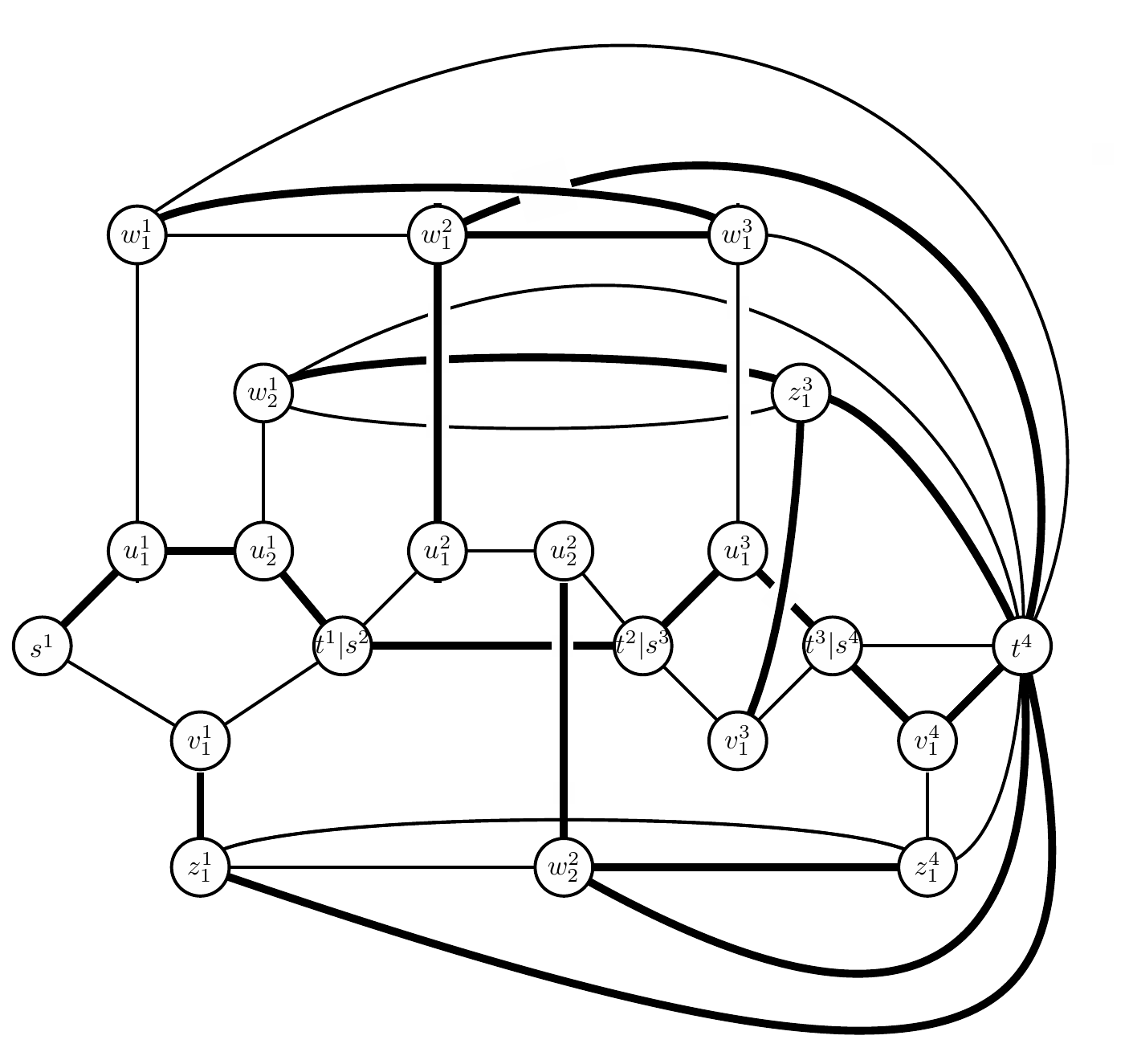}
\caption{The graph corresponding to $\Phi=(x_1\vee x_2\vee x_3)\wedge(x_1\vee \bar{x}_3)\wedge(\bar{x}_1\vee x_2\vee \bar{x}_4)$. Thick and normal edges form modular spanning trees $T_1$ and $T_2$, respectively. Both the assignment $x_1=x_3=1$, $x_2=x_4=0$ corresponding to $T_1$ and the assignment $x_1=x_3=0$, $x_2=x_4=1$ corresponding to $T_2$ are solutions for NAE-SAT.}
\label{fig:graph}
\end{figure}

Now we partition the edge set of $G$ into modules. For every variable $x_j$, if $p_j>0$ then the path $P_j=\{s^ju^j_1,u^j_1u^j_2,\dots,u^j_{p_j}t^j\}$ form a module. Similarly, if $q_j>0$ then the path $N_j=\{s^jv^j_1,v^j_1v^j_2,\dots,v^j_{q_j}t^j\}$ form a module. Finally, the pairs $M^j_k=\{u^j_kw^j_k,w^j_kt_n\}$ form modules of size two for $k=1,\dots,p_j$, and similarly, the pairs $N^j_k=\{v^j_kz^j_k,z^j_kt_n\}$ form modules of size two for $k=1,\dots,q_j$. All the remaining edges of $G$ form modules consisting of a single element.

We claim that $\Phi$ has a truth assignment not setting all literals equally in any clause if and only if $G$ can be partitioned into two modular spanning trees. For the forward direction, let $E=T_1\cup T_2$ be a partition of $E$ into two modular spanning trees. Then
\[
\varphi(x_j)=
\begin{cases}
1 & \ \text{if $p_j>0$ and $P_j\subseteq T_1$, or $p_j=0$ and $s^jt^j\in T_1$,}\\
0 & \ \text{otherwise.}
\end{cases}
\] 
is a truth assignment not setting all literals equally in any clause. To verify this, observe that for a variable $x_j$, if $x_j=1$ then $M^j_k\subseteq T_2$ for $k=1,\dots,p_j$. This follows from the fact that $T_2$ has to span the node $u^j_k$ and $\{u^j_kw^j_k,w^j_kt_n\}$ form a module for $k=1,\dots,p_j$. Similarly, if $x_j=0$ then $N^j_k\subseteq T_2$ for $k=1,\dots,q_j$. Let now $C_i$ be a clause involving variables $x_{j_1},\dots,x_{j_{\ell}}$ and recall the definition of $y^i_{j_1},\dots,y^i_{j_\ell}$. If all the literals in $C_i$ has true value then, by the above observation, the cycle $y^i_{j_1},\dots,y^i_{j_\ell}$ has to lie completely in $T_1$, a contradiction.  If all the literals in $C_i$ has false value then, again by the above observation, the cycle $y^i_{j_1},\dots,y^i_{j_\ell}$  has to lie completely in $T_2$, a contradiction. A similar reasoning shows that $T_2$ also defines a truth assignment not setting all literals equally in any clause.

To see the backward direction, consider a truth assignment $\varphi$ of $\Phi$ not setting all literals equally in any clause. We define the edges of $T_1$ as follows. For each variable $x_j$ with $\varphi(x_j)=1$, we add $P_j$ and $N^j_k$ for $k=1,\dots,q_j$ to $T_1$. For each variable $x_j$ with $\varphi(x_j)=0$, we add $N_j$ and $M^j_k$ for $k=1,\dots,p_j$ to $T_1$. Finally, for each clause $C_i$ involving variables $x_{j_1},\dots,x_{j_{\ell}}$ do the following: for $k=1,\dots,\ell$, if $C_i$ contains the literal $x_{j_k}$ and $\varphi(x_{j_k})=1$ or $C_i$ contains the literal $\bar{x}_{j_k}$ and $\varphi(x_{j_k})=0$, then add the edge $y^i_{j_k}y^i_{j_{k-1}}$ to $T_1$ (indices are meant in a cyclic order). By the assumption that $\varphi$ does not set all literals equally in any clause, this last step will not form cycles in $T_1$. It is not difficult to see that both $T_1$ and its complement $T_2$ are modular spanning trees, thus concluding the proof of the theorem.
\end{proof}

Now Theorem~\ref{thm:main2} is a consequence of the previous results.

\MainTheoremNP*
\begin{proof}
The proof of Theorems~\ref{thm:main} shows that \textsc{PartitionIntoModularBases} can be reduced to \textsc{PartitionIntoCommonBases}. As \textsc{PartitionIntoModularTrees} is a special case of the former problem, the theorem follows by Theorem~\ref{thm:modtrees}.
\end{proof}

As the matroids $M'_\ell,M''_\ell$ given in the proof of Lemma~\ref{lem:help} are graphic, they are linear. If we apply the reduction described in the proof of Theorem~\ref{thm:main} for a graphic matroid $M$, then the matroids $M_1$ and $M_2$ can be obtained from graphic matroids by using direct sums and truncations, hence they are linear as well and an explicit linear representation can be given in polynomial time \cite{lokshtanov2018deterministic}. This in turn implies that \textsc{PartitionIntoCommonBases} is difficult even when both matroids are given by explicit linear representations.

Harvey, Kir\'aly and Lau \cite{harvey2011disjoint} showed that the computational problem of common base packing reduces to the special case where one of the matroids is a partition matroid. Their construction involves the direct sum of $M_1$ and the matroid obtained from the dual of $M_2$ by replacing each element by $k$ parallel elements. This means that if both $M_1$ and $M_2$ are linear, then the common base packing problem reduces to the special case where one of the matroids is a partition matroid and the other one is linear. Concluding these observations, we get the following.

\begin{corollary} \label{cor:main}
The \textsc{PartitionIntoCommonBases} problem includes NP-complete problems even when $r(S)=2|S|$, one of the matroids is a partition matroid and the other is a linear matroid given by an explicit linear representation.
\end{corollary}

\section{Hardness in another special case: the Perfect Even Factor problem} \label{sec:main2}

Let us recall that in the \textsc{PartitionIntoParityBases} problem a matroid $M=(S,\cI)$ is given together with a partition of its ground set into pairs, and the goal is to find a partition of $S$ into parity bases. The aim of this section is to show that \textsc{PartitionIntoParityBases} is difficult for transversal matroids.

First we define the \textsc{$C_{4k+2}$Free2Factor} problem: Given a biparite graph $G=(S,T;E)$, decide if $G$ admits a $2$-factor in which the length of each cycle is a multiple of $4$. This problem was previously studied in \cite{takazawa2017excluded}, where it was shown that the problem is tractable for a special subclass of bipartite graphs. However, for general bipartite graphs \textsc{$C_{4k+2}$Free2Factor} is NP-complete.

\begin{theorem} \label{thm:c4k2}
\textsc{$C_{4k+2}$Free2Factor} is NP-complete.
\end{theorem}
\begin{proof}
We prove by reduction from \textsc{PerfectEvenFactor}: Given a directed graph $D=(V,A)$, decide if there exists a node-disjoint collection of directed cycles of even length covering every node of $D$. This problem was shown to be NP-complete in \cite{cunningham2000combinatorial}.\footnote{In fact, \cite{cunningham2000combinatorial} proves hardness of finding an even factor with maximum size. However, applying the same proof for the 2P2N-SAT problem that is also NP-complete \cite{yoshinaka2005higher,berman04approx}, one get the desired hardness result for the Perfect Even Factor problem.}

Given an instance $D=(V,A)$ of \textsc{PerfectEvenFactor}, we construct a bipartite graph $G=(S,T;E)$ as follows. Let $V'$ and $V''$ denote two copies of $V$. The copies of a node $v\in V$ are denoted by $v'$ and $v''$, respectively. For each $v\in V$, add a path $P_v=\{v'w^v_1,w^v_1w^v_2,w^v_2w^v_3,w^v_3w^v_4,w^v_4v''\}$ of length $5$ between $v'$ and $v''$. For every arc $uv\in A$, add an edge $u'v''$ to the graph. Note that the graph $G=(S,T;E)$ thus obtained is bipartite and, say, $V'\subseteq S$ and $V''\subseteq T$. For a set $F\subseteq A$, let $E_F=\{u'v''\in E:uv\in F\}$ denote the corresponding set of edges in $G$.  We claim that $D$ admits a perfect even factor if and only if $G$ has a $C_{4k+2}$-free $2$-factor (see Figure~\ref{fig:reduction} for the construction). 

\begin{figure}[t!]
\centering
\begin{subfigure}[b]{.35\textwidth}
  \centering
  \includegraphics[width=.8\linewidth]{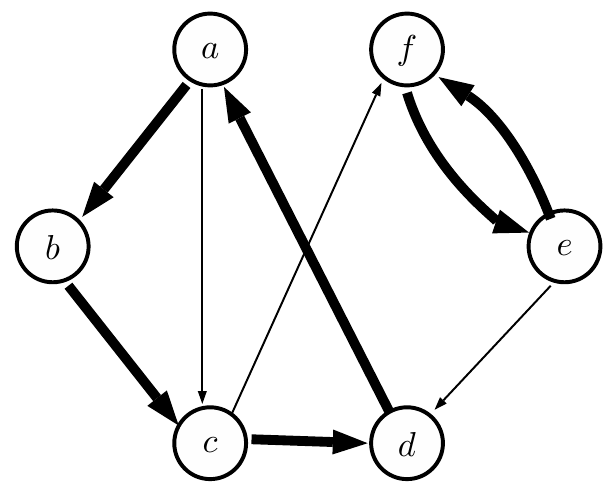}
  \caption{A directed graph $D$ with an even factor $F$ (thick arcs).}
  \label{fig:fac1}
\end{subfigure}\hfill
\begin{subfigure}[b]{.6\textwidth}
  \centering
  \includegraphics[width=.7\linewidth]{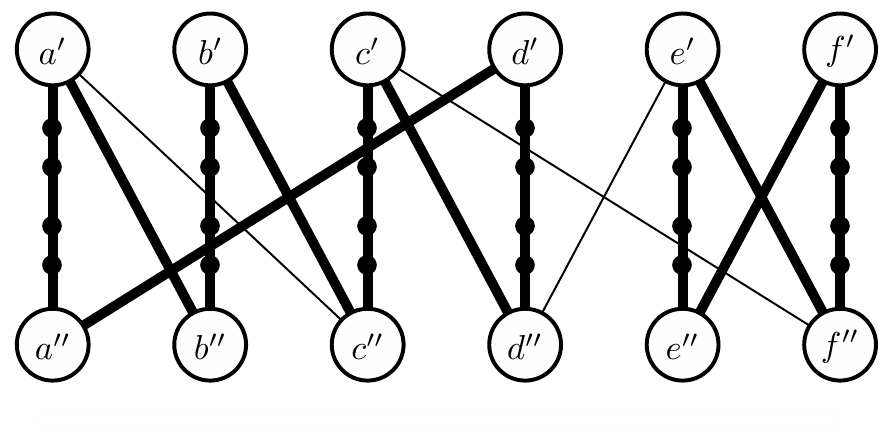}
  \caption{The $2$-factor $N$ corresponding to $F$ (thick edges). Note that the length of each cycle is a multiple of $4$.}
  \label{fig:fac2}
\end{subfigure}
\caption{Reduction from \textsc{PerfectEvenFactor} to \textsc{$C_{4k+2}$Free2Factor}.}
\label{fig:reduction}
\end{figure}

For the forward direction, let $N\subseteq E$ be a $C_{4k+2}$-free $2$-factor of $G$. Due to the presence of nodes having degree $2$, $N$ necessarily contains all the edges of the path $P_v$ for $v\in V$. This implies that the degree of every node $v\in V'\cup V''$ in $N\cap E_A$ is exactly $1$. Hence $N\cap E_A$ corresponds to a subgraph of $D$ in which every node has in- and out-degrees exactly $1$, thus forming a perfect cycle factor of $D$. The length of a cycle in $N$ is six times the original length of the corresponding directed cycle in $D$. As the length of every cycle in $N$ is a multiple of $4$, the corresponding dicycles have even lengths.

To see the backward direction, consider a perfect even factor $F\subseteq A$ in $D$. Then $N=E_F\cup\bigcup_{v\in V} P_v$ is a $C_{4k+2}$-free $2$-factor in $G$. Indeed, it is clear that $N$ is a $2$-factor. The length of a cycle in $N$ is six times the original length of the corresponding directed cycle in $F$. As $F$ was assumed to be an even factor, every cycle in $N$ has a length that is a multiple of $4$, concluding the proof of the theorem. 
\end{proof}
 
Now we show that the \textsc{PartitionIntoParityBases} problem is difficult already for transversal matroids.
 
\begin{theorem} \label{thm:pairednp}
\textsc{PartitionIntoParityBases} includes the \textsc{$C_{4k+2}$Free2Factor} problem.
\end{theorem}
\begin{proof}
Let $G=(S,T;E)$ be an instance of \textsc{$C_{4k+2}$Free2Factor}. We may assume that $|S|=|T|$ as otherwise there is certainly no $2$-factor in $G$. We define a new bipartite graph $G^+=(S'\cup S'',T;E^+)$, where $S'$ and $S''$ are two copies of $S$. The copies of a node $s\in S$ will be denoted by $s'$ and $s''$. For each $st\in E$, we add the edges $s't,s''t$ to $E^+$. Let $P=\{\{s',s''\}:\ s\in S\}$ denote the partitioning of $S'\cup S''$ into pairs where each pair consists of the two copies of a node in $S$. Finally, let $M=(S'\cup S'',\mathcal{I})$ denote the transversal matroid on $S'\cup S''$ defined by $G^+$. We claim that $G$ admits a $C_{4k+2}$-free $2$-factor if and only if the ground set of $M$ can be partitioned into two parity bases (see Figure~\ref{fig:reduction2} for the construction). 

\begin{figure}[t!]
\centering
\begin{subfigure}[t]{.37\textwidth}
  \centering
  \includegraphics[width=.7\linewidth]{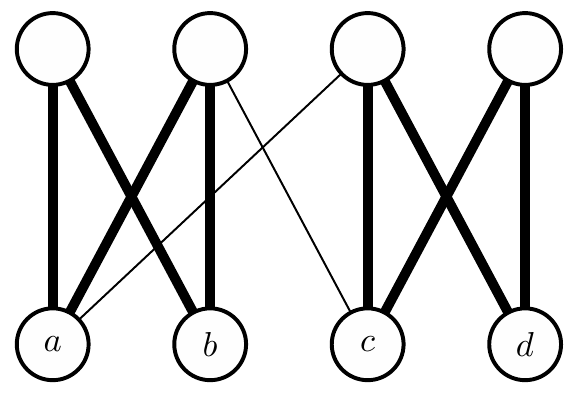}
  \caption{A bipartite graph $G=(S,T;E)$ with a $C_{4k+2}$-free $2$-factor $F$ (thick edges).}
  \label{fig:par1}
\end{subfigure}\hfill
\begin{subfigure}[t]{.6\textwidth}
  \centering
  \includegraphics[width=.9\linewidth]{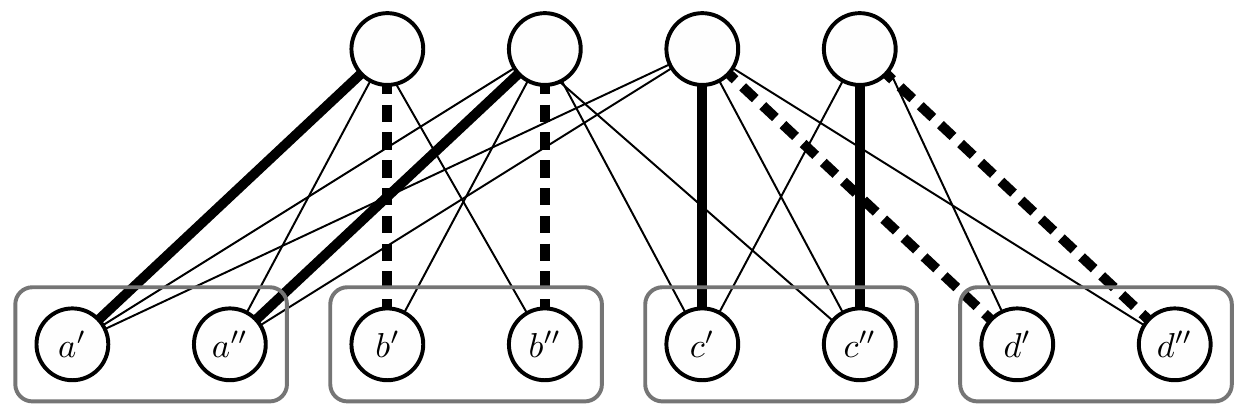}
  \caption{Two matchings (thick and dashed edges) corresponding to $F$.}
  \label{fig:par2}
\end{subfigure}
\caption{Reduction from \textsc{$C_{4k+2}$Free2Factor} to \textsc{PartitionIntoParityBases}.}
\label{fig:reduction2}
\end{figure}

For the forward direction, take a partition of $S'\cup S''$ into parity bases, that is, let $S=S_1\cup S_2$ be a partition of $S$ so that $S'_1\cup S''_1$ and $S'_2\cup S''_2$ are bases of $M$. Then there exist edge-disjoint matchings $N_1$ and $N_2$ in $G^+$ such that $N_i$ covers $S'_i\cup S''_i$ for $i=1,2$. As the two copies $s',s''$ of a node $s\in S$ form a pair, the union of these two matchings contains at most one copy of each edge $st\in E$. Thus, for $i=1,2$, $N_i$ can be naturally identified with a subset $F_i\subseteq E$ in which every node in $T$ has degree exactly $1$ while every node in $S$ has degree either $0$ or $2$. The union of $F_1$ and $F_2$ is then a $2$-factor in which every cycle has a length that is a multiple of $4$.

To see the the backward direction, consider a $C_{4k+2}$-free $2$-factor $F$ of $G$. Let $\{\{s_{i,1}t_{i,1},t_{i,1}s_{i,2},\dots,\allowbreak s_{i,2k_i}t_{i,2k_i}, t_{i,2k_i}s_{i,1}\}: i=1,\dots,q\}$ be the set of cycles appearing in $F$ where $s_{i,j}\in S$ and $t_{i,j}\in T$ for every $i,j$. Let 
\[
S_1=\{s_{i,2j-1}: j=1,\dots,k_i,i=1,\dots,q\}
\]
and
\[
S_2=\{s_{i,2j}: j=1,\dots,k_i,i=1,\dots,q\}.
\]
Then, if the indices are considered in a cyclic order,  
\[
N_1=\{s'_{i,2j-1}t_{i,2j-2},s''_{i,2j-1}t_{i,2j-1}: j=1,\dots,k_i,i=1,\dots,q\}
\]
and 
\[
N_2=\{s'_{i,2j}t_{i,2j-1},s''_{i,2j}t_{i,2j}: j=1,\dots,k_i,i=1,\dots,q\}
\]
are matchings covering $S'_1\cup S''_1$ and $S'_2\cup S''_2$, respectively, concluding the proof of the theorem.
\end{proof}

The proof of Theorem~\ref{thm:c4k2} implies that \textsc{PartitionIntoParityBases} includes NP-complete problems even when restricted to transversal matroids. However, we assumed throughout that the transversal matroid in question is given by a bipartite graph representation. It is not clear whether \textsc{PartitionIntoParityBases} remains difficult if the matroid is given by an explicit linear representation. The authors find it quite unlikely, but it might happen that the problem becomes tractable if a linear representation of the corresponding transversal matroid is also given. However, the randomized polynomial algorithm of \cite{marx2009parameterized} and the gap between the solvability of \textsc{MatroidParity} and \textsc{PartitionIntoParityBases} for transversal matroids given by a bipartite graph suggest that this is not the case, and \textsc{PartitionIntoParityBases} is most probably difficult even if explicit linear representations are given.

\section{Conclusions} \label{sec:conclusions}

In this paper we study a longstanding open problem of matroid theory, the problem of partitioning the ground set of two matroids into common bases. We prove that the problem is difficult, i.e., it requires an exponential number of independence queries in the independence oracle model. We also show that the problem remains intractable for matroids given by explicit linear representations.

The hardness of the general case increases the importance of tractable special cases. The long list of open questions and conjectures that fit in the framework of packing common bases shows that there is still a lot of work to do. For example, one of the simplest cases when one of the matroids is a partition matroid while the other one is graphic remains open.

\section*{Acknowledgement}

The authors are grateful to Andr\'as Frank, Csaba Kir\'aly, Vikt\'oria Kaszanitzky and Lilla T\'othm\'er\'esz for the helpful discussions. The anonymous referees provided several useful comments; the authors gratefully acknowledge their helpful contributions.

Krist\'of B\'erczi was supported by the J\'anos Bolyai Research Fellowship of the Hungarian Academy of Sciences and by the ÚNKP-19-4 New National Excellence Program of the Ministry for Innovation and Technology. Tam\'as Schwarcz was supported by the European Union, co-financed by the European Social Fund (EFOP-3.6.3-VEKOP-16-2017-00002). Projects no. NKFI-128673 and no. ED\_18-1-2019-0030 (Application-specific highly reliable IT solutions) have been implemented with the support provided from the National Research, Development and Innovation Fund of Hungary, financed under the FK\_18 and the Thematic Excellence Programme funding schemes, respectively.

\bibliographystyle{abbrv}
\bibliography{otka}

\end{document}